\documentclass[11pt]{article}
  \RequirePackage{amsthm,amsmath}
     \usepackage{color,latexsym,amsfonts,amssymb}
\usepackage{amsmath}

\usepackage{amsthm}
\usepackage{amssymb}

\newcommand{\beqn}{\begin{eqnarray}}             
\newcommand{\eeqn}{\end{eqnarray}}               
\newcommand{\beq}{\begin{eqnarray*}}             
\newcommand{\eeq}{\end{eqnarray*}}

\def\C{\hbox{\rlap{\kern.24em\raise.1ex\hbox
                  {\vrule height1.3ex width.9pt}}C}}
\def\IR{\hbox{\rlap{I}\kern.16em R}}
\def\Q{\hbox{\rlap{\kern.24em\raise.1ex\hbox
                  {\vrule height1.3ex width.9pt}}Q}}
\def\IM{\hbox{\rlap{I}\kern.16em\rlap{I}M}}
\def\IN{\hbox{\rlap{I}\kern.16em\rlap{I}N}}

\def\E{\hbox{\rlap{I}\kern.16em\rlap{I}E}}
\def\P{\hbox{\rlap{I}\kern.16em\rlap{I}P}}
\def\V{\hbox{\rlap{V}\kern.20em\raise.6ex\hbox{v}}}

\def\one{1\!\kern -.12em {\rm l}}

\newcommand{\ignore}[1]{}{}

\allowdisplaybreaks

\numberwithin{equation}{section}

\newcommand{\be}{\begin{eqnarray}}
\newcommand{\ee}{\end{eqnarray}}
\newcommand{\bestar}{\begin{eqnarray*}}
\newcommand{\eestar}{\end{eqnarray*}}

\newtheorem{thm}{Theorem}[section]
\newtheorem{lemma}{Lemma}[section]

\newtheorem{cor}{Corollary}[section]

\begin{document}

\title{Weak convergence of self-normalized partial sums processes \thanks{Technical Report Series of the Laboratory for Research in Statistics
and Probability, No. 455 - May 2011, Carleton University - University of Ottawa}
}
\author{\bf Mikl\'{o}s Cs\"{o}rg\H{o}\thanks{Research supported by an NSERC Canada Discovery Grant at Carleton University.}\\
\footnotesize School of Mathematics and Statistics \\
\footnotesize Carleton University, 1125
Colonel By Drive, Ottawa, ON K1S 5B6, Canada\\
\bf mcsorgo@math.carleton.ca\\
\\
\bf Zhishui Hu\thanks{Partially supported by NSFC(No.10801122) and RFDP(No.200803581009), and by an NSERC Canada Discovery Grant of M.
Cs\"{o}rg\H{o} at Carleton University. }
\\
 \footnotesize
Department of Statistics and Finance,  School of Management \\
\footnotesize  University of Science and Technology of China, Hefei,
Anhui 230026, China\\
\bf huzs@ustc.edu.cn}

\date{ }
\maketitle

\begin{center}
\begin{minipage}{130mm}
{{\bf Abstract.}
 Let $\{X, X_n, n\geq 1\}$ be a sequence of independent and identically distributed  non-degenerate random variables.
 Put $S_0=0,~S_n = \sum^n_{i=1} X_i$ and $V_n^2=\sum^n_{i=1} X_i^2,~n\ge 1.$
 A  weak convergence theorem is established for the self-normalized partial sums
 processes $\{S_{[nt]}/V_n, 0\le t\le 1\}$ when  $X$ belongs to
the domain of attraction of a stable law with index $\alpha \in (0,2]$.
 The respective limiting distributions of the random variables  ${\max_{1\le i\le
n}|X_i|}/{S_n}$ and ${\max_{1\le i\le n}|X_i|}/{V_n}$ are also obtained under the same condition. }



\end{minipage}
\end{center}

\bigskip

\section{Introduction}

Throughout this paper  $\{X, X_n, n\geq 1\}$ denotes a sequence of independent and identically distributed (i.i.d.)  non-degenerate random
variables. Put $S_0=0$, and
$$
S_n = \sum^n_{i=1} X_i, \quad \bar{X}_n=S_n/n,\quad V_n^2=\sum^n_{i=1} X_i^2, \quad n\geq 1.
$$

The quotient $S_n/V_n$ may be viewed as a self-normalized sum. When $V_n=0$ and hence $S_n=0$,
 we define $S_n/V_n$ to be zero. In terms of
$S_n/V_n$, the classical Student statistic $T_n$ is of the form
 \be
 T_n(X)&=& \frac{(1/\sqrt{n})\sum_{i=1}^n X_i}{\Big((1/(n-1))\sum_{i=1}^n (X_i-\bar{X}_n)^2\Big)^{1/2}}\nonumber\\
&=& \frac{S_n/V_n}{\sqrt{(n-(S_n/V_n)^2)/(n-1)}}\, . \label{deft}
 \ee

If $T_n$ or $S_n/V_n$ has an asymptotic distribution, then so does the other, and they coincide [cf. Efron (1969)]. Throughout,
$\stackrel{d}{\rightarrow}$ will indicate convergence in distribution, or weak convergence, in a given context, while $\stackrel{d}{=}$ will
stand for equality in distribution.

The identification of possible limit distributions of normalized sums $Z_n=(S_n-A_n)/B_n$ for suitably chosen real constants $B_n>0$ and $A_n$,
the description of necessary and sufficient conditions for the distribution function of $X$ such that the distributions of $Z_n$ converge to a
limit, were some of the fundamental problems in the classical theory of limit distributions for identically distributed summands [cf. Gnedenko
and Kolmogorov (1968)]. It is now well-known that $Z_n$ has a non-degenerate asymptotic distribution for some suitably chosen real constants
$A_n$ and $B_n>0$ if and only if $X$ is in the domain of attraction of a stable law with index $\alpha \in (0,2]$. When $\alpha=2$, this is
equivalent to $\ell(x):=EX^2I(|X|\le x)$ being a slowly varying function as $x\rightarrow \infty$, one of the necessary and sufficient analytic
conditions for $Z_n\stackrel{d}{\rightarrow} N(0,1),~n\rightarrow \infty$ [cf. Theorem 1a in Feller (1971), page 313], i.e., for $X$ to be in
the domain of attraction of the normal law, written $X\in \mbox{DAN}$. In this case $A_n$ can be taken as $nEX$ and $B_n=n^{1/2}\ell_X(n)$ with
some function $\ell_X(n)$ that is slowly varying at infinity and determined by the distribution of $X$. Moreover,
$\ell_X(n)=\sqrt{\mbox{Var}~(X)}>0$ if $\mbox{Var}~(X)<\infty$,  and $\ell_X(n)\nearrow \infty$ if $\mbox{Var}~(X)=\infty$. Also, $X$ has
moments of all orders less than $2$, and variance of $X$ is positive, but need not be finite. The function $\ell(x)=EX^2I(|X|\le x)$ being
slowly varying at $\infty$ is equivalent to having $x^2P(|X|>x)=o(\ell(x))$ as $x\rightarrow \infty$, and thus also to having
$Z_n\stackrel{d}{\rightarrow} N(0,1)$ as $n\rightarrow \infty$. In a somewhat similar vein, $Z_n$ having a non-degenerate limiting distribution
when $X$ is in the domain of attraction of a stable law with index $\alpha \in (0,2)$ is equivalent to \bestar 1-F(x)+F(-x)\sim
\frac{2-\alpha}{\alpha} x^{-\alpha} h(x) \eestar and \bestar \frac{1-F(x)}{1-F(x)+F(-x)}\rightarrow p,~~\frac{F(-x)}{1-F(x)+F(-x)}\rightarrow
q\eestar as $x\rightarrow +\infty$,  where $p,q \ge 0, p+q=1$
 and $h(x)$ is slowly varying at $+\infty$ [cf. Theorem 1a in Feller(1971), page 313].
 Also, $X$ has moments of all orders less than $\alpha \in (0,2)$.
 The normalizing constants $A_n$ and $B_n$, in turn, are determined in a rather
complicated way by the slowly varying function $h$.

Now, in view of the results of Gin\'{e}, G\"{o}tze and Mason (1997) and Chistyakov and G\"{o}tze (2004), the problem of finding suitable
constants for $Z_n$ having a non-degenerate limit in distribution when $X$ is in the domain of attraction of a stable law with index $\alpha
\in (0,2]$ is eliminated via establishing the convergence in distribution of the self-normalized sums $S_n/V_n$ or, equivalently, that of
Student's statistic $T_n$, to a non-degenerate limit under the same necessary and sufficient conditions for $X$.

For $X$ symmetric, Griffin and Mason (1991) attribute to Roy Erickson a proof of the fact that having $S_n/V_n\stackrel{d}{\rightarrow}
N(0,1)$, as $n\rightarrow \infty$, does imply that $X \in \mbox{DAN}$. Gin\'{e}, G\"{o}tze and Mason (1997) proved the first such result for
the general case of not necessarily symmetric random variables (cf. their Theorem 3.3), which reads as follows.

\noindent {\bf Theorem A} {\it The following two statements are equivalent: \begin{enumerate}
\item[(a)]  $X\in \mbox{DAN}$ and $EX=0$;
\item[(b)]  $S_n/V_n \stackrel{d}{\rightarrow} N(0,1),~n\rightarrow \infty$.
\end{enumerate}}

Chistyakov and G\"{o}tze (2004), in turn, established the following global result (cf. their Theorem 1.1.) when $X$ has a stable law with index
$\alpha \in (0,2]$.

\noindent {\bf Theorem B} {\it The self-normalized sums $S_n/V_n$ converge weakly as $n\rightarrow \infty$ to a random variable $Z$ such that
$P(|Z|=1)<1$ if and only if
\begin{enumerate}
\item[(i)]  $X$ is in the domain of attraction of a stable law with index $\alpha\in (0,2]$;
\item[(ii)] $EX=0$ if $1< \alpha\le 2$;
\item[(iii)]  if $\alpha=1$, then $X$ is in the domain of attraction of Cauchy's law and Feller's condition holds, that is,
 $\lim\limits_{n\rightarrow \infty}n E\sin(X/a_n)$ exists and is finite, where
$
 a_n=\inf\{x>0: nx^{-2}EX^2I(|X|<x)\le 1\}.
 $
\end{enumerate}}

Moreover, Chistyakov and G\"{o}tze (2004) also proved (cf. their Theorem 1.2)
that the self-normalized sums $S_n/V_n$ converge weakly to a {\it
degenerate limit} $Z$ if and only if $P(|X|>x)$ is a slowly varying function at $+\infty$.

Also, in comparison to the Gin\'{e} {\it et al.} (1997) result of Theorem A
above that concludes the asymptotic  {\it standard} normality of
the sequence of self-normalized sums $S_n/V_n$ if and only if $X\in \mbox{DAN}$
 and $EX=0$, Theorem 1.4 of Chistyakov and G\"{o}tze (2004)
shows that $S_n/V_n$ is asymptotically normal if and only if $S_n/V_n$
is asymptotically {\it standard} normal.

We note in  passing that Theorem 3.3 of Gin\'{e} {\it et al.} (1997)
(cf. Theorem A) and the just mentioned Theorem 1.4 of Chistyakov and
G\"{o}tze (2004) confirm the long-standing conjecture of Logan, Mallows,
 Rice and Shepp (1973) (LMRS for short), stating in particular that ``
$S_n/V_n$ is asymptotically normal if (and perhaps only if) $X$ is in
 the domain of attraction of the normal law" (and $X$ is centered). And in
addition ``It seems worthy of conjecture that the only possible nontrivial
 limiting distributions of $S_n/V_n$ are those obtained when $X$
follows a stable law". Theorems 1.1 and 1.2 of Chistyakov and
G\"{o}tze (2004) (cf. Theorem B above and the paragraph right after) show that
this second part of the long-standing LMRS conjecture also holds if
one interprets nontrivial limit distributions as those, that are not
concentrated at the points $+1$ and $-1$.

The proofs of the results of  Chistyakov and G\"{o}tze (2004) (Theorems 1.1--1.7) are very demanding. They rely heavily on auxiliary results
from probability theory and complex analysis that are proved in their Section 3 on their own.

 As noted by Chistyakov and G\"{o}tze (2004), the ``if" part of their Theorem 1.1
(Theorem B above) follows from the results of LMRS as well, while the ``if" part of their Theorem 1.2 follows from Darling (1952). As described
in LMRS [cf. Lemma 2.4 in Chistyakov and G\"{o}tze (2004); see also Cs\"{o}rg\H{o} and Horv\'{a}th (1988), and S. Cs\"{o}rg\H{o} (1989)], the class
of limiting distributions for $\alpha\in (0,2)$ does not contain Gaussian ones. For more details on the lines of research that in view of LMRS
have led to Theorems A and B above, we refer to the respective introductions of Gin\'{e} {\it et al.} (1997) and Chistyakov and G\"{o}tze
(2004).

Further to the lines of research in hand, it has also become well established in the past twenty or so years that limit theorems for
self-normalized sums $S_n/V_n$ often require fewer, frequently much fewer, moment assumptions than those that are necessary for their classical
analogues [see, e.g. Shao (1997)]. All in all, the asymptotic theory of self-normalized sums has much extended the scope of the classical
theory. For a global overview of these developments we refer to the papers Shao (1998, 2004, 2010), Cs\"org\H{o} {\it et al.}  (2004), Jing
{\it et al.} (2008), and to the book de la Pe\~{n}a, Lai and Shao (2009).

 In view of,
  and inspired by, the Gin\'{e} {\it et al.} (1997) result of Theorem A above,
   Cs\"org\H{o}, Szyszkowicz and Wang (2003) established a self-normalized version of the weak invariance principle
  (sup-norm approximation in probability) under the same necessary and sufficient conditions.  Moreover,
   Cs\"org\H{o} {\it et al.} (2008) succeed in
  extending the latter weak invariance principle via weighted sup-norm and $L_p$-approximations, $0<p<\infty$, in probability, again
  under the same necessary and sufficient conditions.  In particular, for dealing with sup-norm approximations, let $Q$ be the class of positive
  functions $q(t)$ on $(0,1]$, i.e., $\inf_{\delta\leq t\leq 1} q(t)>0$ for $0<\delta<1$, which are nondecreasing near zero, and let
$$
I(q,c) := \int^1_{0+} t^{-1}\exp(-cq^2(t)/t)dt, \quad 0<c<\infty.
$$
Then [cf.\ Corollary 3 in Cs\"org\H{o} {\it et al.} (2008)], {\it on assuming that $q\in Q$, the following two statements are equivalent}:

\noindent (a) $X \in \mbox{DAN}$ {\it and} $EX=0$;

\noindent (b) {\it On an appropriate probability space for $X, X_1,X_2,\ldots$, one can construct a standard Wiener process $\{W(s),\, 0\leq
s<\infty\}$ so that, as $n\to\infty$,
\begin{equation}
\sup_{0<t\leq 1} \Big| S_{[nt]}/V_n - W(nt)/n^{1/2}\Big|\Big/ q(t) = o_P(1) \label{aeq1.3}
\end{equation}
if and only if $I(q,c) < \infty$ for all $c>0$. }

\medskip
With $q(t) = 1$ on $(0,1]$, this is Theorem 1 of Cs\"org\H{o} {\it et al.} (2003), and when $\sigma^2 = EX^2<\infty$, then  (\ref{aeq1.3})
combined with Kolmogorov's law of large numbers results in the classical weak invariance principle that in turn yields Donsker's classical
functional CLT.

This work was inspired by the Chistyakov and G\"{o}tze (2004) result of Theorem B above. Our main aim is to identify the limiting distribution
in the latter theorem under the same necessary and sufficient conditions in terms of weak convergence on $D[0,1]$ (cf. Theorem \ref{th1}).
 Our auxiliary Lemma \ref{prop1} may be viewed as a scalar normalized version of Theorem B (Theorem 1.1 of Chistyakov and G\"{o}tze (2004)).

\section{Main results}

An ${\bf R}$-valued stochastic process $\{X(t), t\ge 0\}$ is called a {\it L\'{e}vy process}, if the following four conditions are satisfied:

\begin{enumerate}
\item[(1)] it starts at the origin, i.e. $X(0)=0$ a.s.;
\item[(2)] it has independent increments, that is, for any choice of $n\ge 1$ and $0\le t_0<t_1<\cdots<t_n$, the random variables
$X(t_0),~X(t_1)-X(t_0),\cdots, X(t_n)-X(t_{n-1})$ are independent;
\item[(3)] it is time homogeneous, that is, the distribution of $\{X(t+s)-X(s): t\ge 0\}$ does not depends on $s$;
\item[(4)] as a function of $t$, $X(t,\omega)$ is a.s. right-continuous with left-hand limits.
\end{enumerate}

A {\it L\'{e}vy process} $\{X(t), t\ge 0\}$ is called {\it $\alpha$-stable} (with index $\alpha \in (0,2]$) if for any $a>0$, there exists some
$c\in {\bf R}$ such that $\{X(at)\}\stackrel{d}{=}\{a^{1/\alpha} X(t)+ct\}$. If $\{X(t), t\ge 0\}$ is an {\it  $\alpha$-stable L\'{e}vy
process}, then for any $t\ge 0$, ~$X(t)$ has a stable distribution. For more details about L\'{e}vy  and $\alpha$-stable L\'{e}vy processes, we
refer to Bertoin (1996) and Sato (1999).

It is well known that $G$ is a stable distribution with index $\alpha \in (0,2]$ if and only if  its characteristic function
$f(t)=\int_{-\infty}^{\infty} e^{itx}dG(x)$ admits the representation (see for instance Feller(1971))\be
f(t)=\left\{\begin{array}{ll}\exp\Big\{i\gamma t+c|t|^{ \alpha}\frac{\Gamma(3-\alpha)}{\alpha(\alpha-1)}\Big[\cos \frac{\pi\alpha}{2}+
i(p-q)\frac{t}{|t|}\sin
\frac{\pi\alpha}{2} \Big]\Big\}, & \mbox{~if~}\alpha \ne 1;\\
\exp\Big\{i\gamma t-c|t|\Big[ \frac{\pi}{2}+ i(p-q)\frac{t}{|t|} \log |t|]\Big\}, & \mbox{~if~}\alpha=1,
\end{array}\right.\label{hua1}\ee
where  $c, p,q, \gamma$  are real constants with $c, p,q \ge 0,~p+q=1$. Write $G\sim S(\alpha, \gamma, c, p, q)$ and, as in Theorem B, let
\bestar a_n=\inf\{x>0: nx^{-2}EX^2I(|X|<x)\le 1\}.\eestar

The following result is our main theorem.

\begin{thm} \label{th1} Let $X, X_1,X_2,\cdots$ be a sequence of i.i.d. non-degenerate random variables  and
let $G\sim S(\alpha, \gamma, c, p, q)$. If $X$ is in the domain of attraction of $G$ of index $\alpha\in (0,2]$, with $EX=0$ if $1<\alpha\le 2$
and
 $\lim\limits_{n\rightarrow \infty}n E\sin(X/a_n)$ exists and is finite if $\alpha=1$, then, as $n \rightarrow \infty$, we have
 $$\frac{S_{[nt]}}{V_n} \stackrel{d}{\rightarrow} \frac{X(t)}{\sqrt{[X]_1}}$$  on $D[0,1]$, equipped with the Skorokhod $J_1$ topology,
where $X(t)$ is an $\alpha$-stable L\'{e}vy process of index $\alpha\in (0,2]$ on $[0,1]$,~ $X(1) \sim S(\alpha, \gamma', 1, p, q)$ with
$\gamma'=0$ if $\alpha\ne 1$ and
 $\gamma'=\lim\limits_{n\rightarrow \infty}n E\sin(X/a_n)$ if $\alpha=1$,
 and $[X]_t$ is the quadratic variation of $X(t)$.
\end{thm}

\vskip 0.3cm

 When $\alpha=2$, $G$ is a normal distribution,~ $X(1)\stackrel{d}{=} N(0, 1)$ and $[X]_1=1$.
 Consequently, ${X(t)}/{\sqrt{[X]_1}}$ is a standard Brownian motion
 and thus  we obtain the weak convergence of $S_{[nt]}/V_n$ to a Brownian motion as in (c) of  Theorem 1 of  Cs\"{o}rg\H{o} {\it et al.} (2003)
 [see also our lines right after (\ref{aeq1.3})].

Consider now the sequence $T_{n,t}$ of Student processes in $t\in [0,1]$ on $D[0,1]$, defined as \be \{T_{n,t}(X),~0\le t\le 1\}:&=&\left\{
 \frac{(1/\sqrt{n})\sum_{i=1}^{[nt]} X_i}{\Big((1/(n-1))\sum_{i=1}^n (X_i-\bar{X}_n)^2\Big)^{1/2}},~0\le t\le 1\right\}\nonumber\\
&=& \left\{\frac{S_{[nt]}/V_n}{\sqrt{(n-(S_n/V_n)^2)/(n-1)}},~0\le t\le 1\right\}\, .\ee

Clearly, $T_{n,1}(X)=T_n(X)$, with the latter as in (\ref{deft}). Clearly also, in view of Theorem \ref{th1}, the same result continues to hold
true under the same conditions for the Student process $T_{n,t}$ as well, i.e., Theorem \ref{th1} can be restated in terms of the latter
process. Moreover, if $1< \alpha \le 2$, then $EX =: \mu$ exists and the following corollary obtains.

\begin{cor} \label{cor1} Let $X, X_1,X_2,\cdots$ be a sequence of i.i.d. non-degenerate random variables  and
let $G\sim S(\alpha, \gamma, c, p, q)$. If $X$ is in the domain of attraction of $G$ of index $\alpha\in (1,2]$, then, as $n \rightarrow
\infty$, we have
 $$T_{n,t}(X-\mu)= \frac{(1/\sqrt{n})\sum_{i=1}^{[nt]} (X_i-\mu)}{\Big((1/(n-1))\sum_{i=1}^n (X_i-\bar{X}_n)^2\Big)^{1/2}}
 \stackrel{d}{\rightarrow} \frac{X(t)}{\sqrt{[X]_1}}$$
 on $D[0,1]$, equipped with the Skorokhod $J_1$ topology,
where $X(t)$ is an $\alpha$-stable L\'{e}vy process of index $\alpha\in (1,2]$ on $[0,1]$,~ $X(1) \sim S(\alpha, 0, 1, p, q)$,
 and $[X]_t$ is the quadratic variation of $X(t)$.
\end{cor}

As noted earlier, with $\alpha=2$, ${X(t)}/{\sqrt{[X]_1}}$ is a standard Brownian motion. Moreover, in the latter case, we have
 $(X-\mu) \in \mbox{DAN}$ and this, in turn, is equivalent to having (\ref{aeq1.3}) with $T_{n,t}(X-\mu)$ as well, instead of $S_{[nt]}/V_n$
 [cf. Corollary 3.5 in Cs\"{o}rg\H{o} {\it et al.} (2004)].

 Corollary \ref{cor1} extends the feasibility of the use of the Student process $T_{n,t}(X-\mu)$ for constructing functional asymptotic
 confidence intervals for $\mu$, along the lines of Martsynyuk (2009a, b), beyond $X-\mu$ being in the domain of attraction of the normal law.

Via the proof of Theorem \ref{th1} we can also get a weak convergence result when $X$ belongs to the domain of partial attraction of an
 infinitely divisible law (cf. Feller (1971), page 590).
\begin{thm} \label{prop2}
Let $X(t)$ be a L\'{e}vy process with  $[X]_1\ne 0$, where $[X]_t$ is the quadratic variation of $X(t)$. If there exist some positive constants
$\{b_n\}$ and some subsequence $\{m_n\}$, where $m_n\rightarrow \infty$ as $n\rightarrow \infty$, such that
$S_{m_n}/b_n\stackrel{d}{\rightarrow} X(1)$ as $n \rightarrow \infty$, then ${S_{[m_nt]}}/{V_{m_n}} \stackrel{d}{\rightarrow}
{X(t)}/{\sqrt{[X]_1}}$ on $D[0,1]$, equipped with the Skorokhod $J_1$ topology.
\end{thm}

As will be seen, in the proof of  Theorem \ref{th1}  and \ref{prop2}  we make use of a weak convergence result for sums of  {\it exchangeable
random variables}. For any finite or infinite sequence $\xi=(\xi_1,\xi_2,\cdots)$, we say $\xi$ is exchangeable if
$$(\xi_{k_1},\xi_{k_2},\cdots)\stackrel{d}{=}(\xi_1,\xi_2,\cdots)$$
for any finite permutation $(k_1,k_2,\cdots)$ of ${\bf N}$. A {\it process} $X(t)$ on $[0,1]$ is {\it exchangeable} if it is continuous in
probability with $X_0=0$ and has exchangeable increments over any set of disjoint intervals of equal length. Clearly, a L\'{e}vy process is
exchangeable.

 By using
the  notion of exchangeability, we can get the following corollary from the proof of Theorem \ref{th1}.

\begin{cor} \label{th2}
Let $X,X_1,X_2, \cdots$ and $G$ be as in Theorem \ref{th1}. If $X$ is in the domain of attraction of $G$ of index $\alpha\in (0,2]$, with
$EX=0$ if $1<\alpha\le 2$ and
 $\lim\limits_{n\rightarrow \infty}n E\sin(X/a_n)$ exists and is finite if $\alpha=1$, then, as $n \rightarrow \infty$,
\be \Big(\frac{S_n}{a_n}, \frac{V_n^2}{a_n^2}, \frac{\max_{1\le i\le n}|X_i|}{a_n}\Big)\stackrel{d}{\rightarrow} (X(1), [X]_1, J), \label{3in1}
\ee
 where,  with $\Delta X(t):=X(t)-X(t-)$, ~$J=\max \{|\Delta X(t)|: 0\le t\le 1\}$ is the biggest jump of $X(t)$ on $[0,1]$, where, as in Theorem
 \ref{th1}, $X(t)$ is an $\alpha$-stable L\'{e}vy process with index $\alpha\in (0,2]$ on $[0,1]$, $X(1) \sim S(\alpha, \gamma', 1, p, q)$
as specified in Theorem \ref{th1}, and $[X]_t$ is the quadratic variation of $X(t)$.
\end{cor}
We note in passing that, under the conditions of Corollary \ref{th2}, the joint convergence in distribution as $n \rightarrow \infty$
 \be
\Big(\frac{S_n}{a_n}, \frac{V_n^2}{a_n^2}\Big)\stackrel{d}{\rightarrow} (X(1), [X]_1) \label{last1}
 \ee
amounts to {\it an extension of Raikov's theorem from $X\in \mbox{DAN}$ to $X$ being in the domain of attraction of $G$ of index $\alpha \in
(0,2]$}. When $\alpha=2$, i.e., when $X\in \mbox{DAN}$, the statement of (\ref{last1}) reduces to Raikov's theorem in terms of having
$\Big(\frac{S_n}{a_n}, \frac{V_n^2}{a_n^2}\Big)\stackrel{d}{\rightarrow} (N(0,1), 1)$ as $n \rightarrow \infty$ (cf. Lemma 3.2 in Gin\'{e} {\it
et al.} (1997)).

As a consequence of Corollary \ref{th2}, {\it under the same conditions}, as $n\rightarrow \infty$, we have \be
 && \frac{\max_{1\le i\le n}|X_i|}{S_n} \stackrel{d}{\rightarrow}
 \frac{J}{X(1)}, \label{last2}\ee
 and
 \be
 && \frac{\max_{1\le i\le n}|X_i|}{V_n} \stackrel{d}{\rightarrow}
 \frac{J}{\sqrt{[X]_1}}. \label{last3}
 \ee

In case of $\alpha=2,~G$ is a normal distribution, $X \in \mbox{DAN}$ with $EX=0$, and $X(t)/\sqrt{[X]_1}$ is a standard Brownian motion.
Consequently, $J$ in Corollary \ref{th2} is zero and, as $n \rightarrow \infty$, we arrive at the conclusion that when
 $X\in \hbox{DAN}$ and $EX=0$, then the respective conclusions of (\ref{last2}) and (\ref{last3}) reduce to ${\max_{1\le i\le
n}|X_i|}/{|S_n|}\stackrel{P}{\rightarrow} 0$ and ${\max_{1\le i\le n}|X_i|}/{V_n}\stackrel{P}{\rightarrow} 0$.Kesten and Maller (1994, Theorem
3.1) proved that ${\max_{1\le i\le n}|X_i|}/{|S_n|}\stackrel{P}{\rightarrow} 0$ is equivalent to having
$$\frac{x|EXI(|X|\le x)|+EX^2I(|X|\le x)}{x^2P(|X|>x)}\rightarrow \infty,$$
and O'Brien(1980) showed that ${\max_{1\le i\le n}|X_i|}/{V_n}\stackrel{P}{\rightarrow} 0$ is equivalent to $X\in \mbox{DAN}$.

For $X$ in the domain of attraction of a stable law with index $\alpha \in (0,2)$, Darling (1952) studied the asymptotic behavior of
$S_n/\max_{1\le i\le n} |X_i|$  and derived  the characteristic function of the appropriate limit distribution. Horv\'{a}th and Shao (1996)
established a large deviation and, consequently,  the law of the iterated logarithm for $S_n/\max_{1\le i\le n} |X_i|$ under the same condition
for $X$ symmetric.

Proofs of Theorems \ref{th1}, \ref{prop2} and Corollary \ref{th2}  are given in Section \ref{sect2}.

\section{Proofs} \label{sect2}

Before proving Theorem \ref{th1}, we conclude the following lemma.
\begin{lemma} \label{prop1} Let $G\sim S(\alpha, \gamma, c, p, q)$ with index $\alpha \in (0,2]$ and let $Y_{\alpha}$ be a random variable
associated with this distribution. If there exist some positive constants $\{A_n\}$  satisfying $S_n/A_n\stackrel{d}{\rightarrow} Y_{\alpha}$
as $n \rightarrow \infty$, then
\begin{enumerate}
\item[(1)] $X$ is in the domain of attraction of $G$,
\item[(2)] $EX=0$ if $1<\alpha\le 2$,  and $\lim\limits_{n\rightarrow \infty}n E\sin(X/a_n)$ exists and is finite if $\alpha=1$.
\end{enumerate}
Conversely, if the above conditions (1) and (2) hold, then  \bestar S_n/a_n\stackrel{d}{\rightarrow} Y_{\alpha}',~~\alpha\in (0,2], \eestar
where $Y_{\alpha}'$ is a random variable with distribution
 $G' \sim S(\alpha,\gamma', 1, p, q)$,  with
$\gamma'=0$ if $\alpha\ne 1$ and  $\gamma'=\lim\limits_{n\rightarrow \infty}n E\sin(X/a_n)$ if $\alpha=1$.
\end{lemma}

\begin{proof}[{\bf Proof.}] If $\alpha=2$, then $G$ is a normal distribution and the conclusion with $X \in \mbox{DAN}$ and $EX=0$  is clear.

If $0<\alpha<2$, then $X$ belongs to the domain of attraction of a stable law $G$  with the  characteristic function $f(t)$ as in (\ref{hua1})
if and only if (cf. Theorem 2 in Feller (1971), page 577)\bestar \ell(x)=EX^2I(|X|\le x)=x^{2-\alpha}L(x), ~x \rightarrow \infty,\eestar and
\bestar \frac{P(X>x)}{P(|X|>x)}\rightarrow p,~ x\rightarrow \infty,\eestar
 where $L(x)$ is a slowly varying function at infinity. In this case, as $n \rightarrow \infty$, we have (cf.  Theorem 3 in Feller(1971), page 580)
 \be \frac{S_n}{a_n}-b_n \stackrel{d}{\rightarrow} \tilde{Y}_{\alpha} \mbox{~with distribution~} \tilde{G},~~\alpha \in (0,2), \label{hua2}\ee
where  \bestar b_n=\left\{\begin{array}{ll} (n/a_n)EX,& \mbox{~if~} 1<\alpha<2;\\
n E\sin (X/a_n), &  \mbox{~if~} \alpha=1;\\
 0, &  \mbox{~if~} 0<\alpha<1, \end{array}
 \right.\eestar
and $\tilde{G}\sim S(\alpha,0, 1, p, q)$. Thus if (2) holds, then, as $n \rightarrow \infty$, we have \bestar \frac{S_n}{a_n}
\stackrel{d}{\rightarrow} {Y}_{\alpha}' \mbox{~with distribution~} G',~~\alpha \in (0,2), \eestar where $G' \sim S(\alpha,\gamma', 1, p, q)$
with $\gamma'=0$ if $\alpha\ne 1$ and $\gamma'=\lim\limits_{n\rightarrow \infty}n E\sin(X/a_n)$ if $\alpha=1$.

If, as $n \rightarrow \infty$, there exists some positive constants $\{A_n\}$  satisfying $S_n/A_n\stackrel{d}{\rightarrow} Y_{\alpha}$ with
distribution $G$ with index $\alpha\in (0,2)$, then (1) holds. Hence (\ref{hua2}) is also true. Consequently, by Theorem 1.14 in Petrov (1995),
we have  $b_n\rightarrow b$ for some real constant $b$, as $n \rightarrow \infty$. Thus if $\alpha=1$, then $\lim\limits_{n\rightarrow \infty}n
E\sin(X/a_n)$ exists and is finite, and if $1<\alpha<2$, since in this case $n/a_n=na_n^{-\alpha}L(a_n)(a_n^{\alpha-1}/L(a_n))\sim
a_n^{\alpha-1}/L(a_n) \rightarrow \infty$ as $n\rightarrow \infty$,  we have $EX=0$.

Proof of Lemma \ref{prop1} is now complete.
\end{proof}

\vskip 0.5cm

 \begin{proof}[{\bf Proof of Theorem \ref{th1}.}]  Since $X(t)$ is a L\'{e}vy process, we have the L\'{e}vy-It\^{o} decomposition (see for instance
Corollary 15.7 in Kallenberg (2002) ) \be X(t)=b t+\sigma W(t)+\int_0^t\int_{|x|\le 1} x(\eta-E\eta)(ds,dx)+\int_0^t\int_{|x|>1} x\eta(ds,dx),
\label{hu1}\ee for some $b\in {\bf R}, \sigma\ge 0$,~where $W(t)$ is a Brownian motion independent of $\eta$,  and $\eta=\sum_t \delta_{t,
\Delta X_t}$ is a
 Poisson process on $(0,\infty)\times ({\bf
R}\setminus\{0\})$ with  $E\eta=\lambda \otimes \nu$,  where $\Delta X_t=X_t-X_{t-}$ is the jump of $X$ at time $t$,~ $\lambda$ is the Lebesgue
measure on $(0,\infty)$ and $\nu$ is some measure on ${\bf R}\setminus\{0\}$ with $\int (x^2\wedge 1) \nu(dx)<\infty$. The quadratic variation
of $X(t)$ is (cf. Corollary 26.15 in Kanllenberg (2002))
 \be [X]_t=\sigma^2 t+\sum_{s\le t} (\Delta X_s)^2. \label{add3}\ee

Noting that a L\'{e}vy Process is   exchangeable, by Theorem 2.1 of Kallenberg (1973) (or Theorem 16.21 in Kallenberg (2002)), $X(t)$ has a
version $X'(t)$,  with representation \be X'(t)=b't+\sigma' {B}(t)+\sum_{j}\beta_j (I(\tau_j\le t)-t), ~~t\in [0,1],\label{hu2} \ee in the
sense of a.s. uniform convergence, where
\begin{enumerate}
\item[(1)] $ b'=X(1),~\sigma' \ge 0,~\beta_1\le \beta_3 \le \cdots\le 0\le \cdots \le \beta_4\le \beta_2$
are random variables with $\sum_j \beta_j^2<\infty,~a.s.$,
\item[(2)] ${B}(t)$ is a Brownian bridge on $[0,1]$,
\item[(3)]  $\tau_1,\tau_2,\cdots$ are independent and uniformly distributed random variables on $[0,1]$,
\end{enumerate}
and the three groups (1)-(3) of random elements are independent. $X(t)$ has a version $X'(t)$ means that for any $t \in [0,1]$, ~$X(t)=X'(t)~
~a.s.$ But since both  $X(t)$ and $X'(t)$ are right continuous, we have $$P(X(t)=X'(t)~\mbox{for all~} t\in [0,1])=1.$$

Thus we may say that $X(t)\equiv X'(t)$ on $[0,1]$.  By (\ref{hu2}), we  get that $(\beta_1,\beta_2,\cdots)$ are the sizes of the jumps  of
$\{X(t), t\in [0,1]\}$ and $(\tau_1,\tau_2,\cdots)$ are the related jump times. Thus $\eta=\sum_j \delta_{\tau_j,\beta_j}$ on $(0,1]\times
({\bf R}\setminus\{0\})$ and, by (\ref{add3}),
$$[X]_1=\sigma^2 +\sum_{s\le 1} (\Delta X_s)^2=\sigma^2 +\sum_j \beta_j^2.$$  We are to see now that we also have \be \sigma'=\sigma. \label{add2}\ee
Write \bestar &&X^n(t)=b t+\sigma W(t)+\int_0^t\int_{1/n<|x|\le 1}
x(\eta-E\eta)(ds,dx)+\int_0^t\int_{|x|>1} x\eta(ds,dx)\\
&&~~~~~~~=\tilde{b}_n t+\sigma \tilde{B}(t)+\sum_{|\beta_j|>1/n} \beta_jI(\tau_j\le t),~~n\ge 1, \eestar where $\tilde{b}_n=b+\sigma W(1)-\int
xI(1/n<|x|\le 1)\nu(dx)$ and $\tilde{B}(t)=W(t)-tW(1)$ is a Brownian bridge. Noting that $W(1)$ and $\{\tilde{B}(t)\}$ are independent,
$X^n(t)$ is also an exchangeable process for each $n\ge 1$. From the proof of Theorem 15.4 in Kallenberg (2002), we have $$E\sup_{0\le s\le 1}
(X(s)-X^n(s))^2\rightarrow 0,~~n \rightarrow \infty.$$ Thus, as $n \rightarrow \infty$, $X^n(t)\stackrel{d}{\rightarrow} X(t)$ on $D(0,1)$ with
the Skorokhod $J_1$ topology. Then, by Theorem 3.8 in Kallenberg (2005), as $n \rightarrow \infty$, we have \bestar
\sigma^2+\sum_{|\beta_j|>1/n}\beta_j^2\stackrel{vd}{\longrightarrow}
\sigma'^{2}+\sum_j\beta_j^2,\eestar
where $\stackrel{vd}{\longrightarrow}$ means convergence in distribution with respect to the vague topology.
 Hence $\sigma'^2=\sigma^2$, and (\ref{add2}) holds.

By Lemma \ref{prop1},  $S_n/a_n \stackrel{d}{\rightarrow} X(1)$. Hence, by Theorem 16.14 in Kallenberg (2002), we have $S_{[nt]}/a_n
\stackrel{d}{\rightarrow} X(t)$  on $D(0,1)$ with the Skorokhod $J_1$ topology. By noting that $\{X_i/a_n, i=1,\cdots, n\}$ are exchangeable
random variables for each $n$, and by using Theorems 3.8 and 3.13 in Kallenberg (2005), as $n \rightarrow \infty$, we have \be \Big(\frac{S_n}{a_n},\,
\sum_{i=1}^n \frac{X_i^2}{a_n^2},\, \sum_{i=1}^n \delta_{X_i/a_n} \Big)\stackrel{vd}{\longrightarrow} (X(1),\, [X]_1,\, \sum_j\delta_{\beta_j})
~~\mbox{in}~~ {\bf R}\times {\bf R}_+\times \mathcal{N}({\bf{R}}\setminus\{0\}),
~~~~\label{all1} \ee
where $\mathcal{N}({\bf{R}}\setminus\{0\})$ is the space of integer-valued measures on ${\bf{R}}\setminus\{0\}$ endowed with the vague topology.
%
Hence \bestar \Big(\frac{S_n}{V_n}, \, \sum_{i=1}^n \frac{X_i^2}{V_n^2},  \, \sum_{i=1}^n\delta_{X_i/V_n} \Big)\stackrel{vd}{\longrightarrow}
\Big(\frac{X(1)}{\sqrt{[X]_1}}, \, 1, \,  \sum_j\delta_{\beta_j/\sqrt{[X]_1}}\Big)
~~\mbox{in}~~ {\bf R}\times {\bf R}_+\times \mathcal{N}({\bf{R}}\setminus\{0\}). \eestar Since $\{X_i/V_n,
i=1,\cdots, n\}$ are exchangeable for each $n$, by Theorems 3.8 and 3.13 in Kallenberg (2005), we have \beq
\frac{S_{[nt]}}{V_n}\stackrel{d}{\rightarrow} \frac{X(t)}{\sqrt{[X]_1}}\eeq
on $D[0,1]$, equipped with the Skorokhod $J_1$ topology.
\end{proof}

\vskip 0.5cm

 \begin{proof}[{\bf  Proof of Theorem \ref{prop2}.}]  It is similar to the proof of Theorem \ref{th1} with only minor changes. Hence we omit the
details.
\end{proof}

\vskip 0.5cm
 \begin{proof}[\bf Proof of Corollary \ref{th2}.]  Note that (\ref{all1})
 is equivalent to (see remarks below Theorem 2.2 of Kallenberg (1973)) \be \Big(\frac{S_n}{a_n}, \sum_{i=1}^n
\frac{X_i^2}{a_n^2}, \frac{{X}_{n1}}{a_n},\frac{{X}_{n2}}{a_n},\cdots \Big)\stackrel{d}{\rightarrow} (X(1), [X]_1, \beta_1,
\beta_2,\cdots) \mbox{~in ~} {\bf R}^{\infty}, \label{add1}\ee where ${X}_{n1}\le {X}_{n3}\le\cdots \le 0\le \cdots {X}_{n4}\le {X}_{n2}$ are
obtained by ordering $\{X_i, 1\le i\le n\}\cup \{\tilde{X}_i, i>n\}$ with $\tilde{X}_i\equiv 0,~i=1,2,\cdots$.
Now the conclusion of (\ref{3in1}) follows directly from (\ref{add1}).
\end{proof}


\begin{thebibliography}{13}
\bibitem{} Bertoin, J. (1996) {\it L\'{e}vy Processes.} Cambridge University Press.


\bibitem{}
Chistyakov, G. P. and G\"{o}tze, F. (2004). Limit distributions of Studentized means. {\it Ann. Probab.} {\bf 32}, 28-77.


\bibitem{} Cs\"{o}rg\H{o}, M. and Horv\'{a}th, L. (1988)
Asymptotic representations of self-normalized sums. {\it Probab. Math. Statist.} {\bf 9}(1), 15-24.

\bibitem{} Cs\"{o}rg\H{o}, M.,  Szyszkowicz, B. and Wang, Q.  (2003)
Donsker's theorem for self-normalized partial sums processes. {\it Ann. Probab.} {\bf 31}, 1228-1240.

\bibitem{}Cs\"{o}rg\H{o}, M.,  Szyszkowicz, B. and Wang, Q.\ (2004)
On Weighted Approximations and Strong Limit Theorems for Self-normalized Partial Sums Processes.  In {\it Asymptotic Methods in Stochastics},
(L.~Horv\'ath, B.~Szyszkowicz, eds.), 489--521, Fields Inst.\ Commun. {\bf 44}, Amer.\ Math.\ Soc., Providence RI.


\bibitem{} Cs\"{o}rg\H{o}, M.,  Szyszkowicz, B. and Wang, Q.\ (2008)
On weighted approximations in $D[0,1]$ with application to self-normalized partial sum processes. {\it Acta Mathematica Hungarica} {\bf 121
(4)} 307--332.

\bibitem{} Cs\"{o}rg\H{o}, S. (1989) Notes on extreme and self-normalised sums from the domain of attraction of a stable law.
{\it J. London Math. Soc.} {\bf 39}, 369-384.

\bibitem{} Darling, D. A. (1952) The influence of the maximum term in the addition of independent random variables.
{\it Trans. Amer. Math. Soc.} {\bf 73},  95-107.

\bibitem{} de la Pe\~{n}a, V. H., Lai, T. L. and  Shao, Q.-M.(2009).
{\it Self-normalized processes. Limit theory and statistical applications}. Probability and its Applications (New York). Springer-Verlag,
Berlin.


\bibitem{}  Efron, B. (1969) Student's t-test under symmetry conditions. {\it J. Amer. Statist. Assoc.} {\bf 64}, 1278-1302.

\bibitem{} Feller, W. (1971) {\it An Introduction to Probability Theory and Its Applications.} Vol. 2,  Wiley, New York.




\bibitem{}  Gin\'{e}, E.,  G\"{o}tze, F. and Mason D. (1997)
When is the Student $t$-statistic asymptotically standard normal? {\it Ann. Probab.} {\bf 25},  1514-1531.

\bibitem{} Gnedenko,  B. V.  and  Kolmogorov, A. N. (1968)
{\it Limit distributions for sums of independent random variables.}  Addison-Wesley, Cambridge

\bibitem{}
Griffin, P. S. and Mason, D. M. (1991) On the asymptotic normality of self-normalized sums. {\it Proc. Cambridge Phil. Soc.} {\bf 109},
597-610.


\bibitem{}  Horv\'{a}th, L. and Shao, Q.-M. (1996) Large deviations and law of the iterated logarithm for
partial sums normalized by the largest absolute observation. {\it Ann. Probab.} {\bf 24}, 1368-1387

\bibitem{} Jing, B.-Y., Shao, Q.-M. and Zhou, W.\ (2008)
Towards a universal self-normalized moderate deviation. {\it Trans.\ Amer.\ Math.\ Soc.} {\bf 360} 4263-4285.



\bibitem{}
Kallenberg, O. (1973)  Canonical representations and convergence criteria for
  processes with interchangeable increments. {\it Z. Wahrsch.
 Verw. Geb.} {\bf 27}, 23-36.

\bibitem{}
Kallenberg, O. (2002) {\it Foundations of modern probability}.   New York, Springer-Verlag.

\bibitem{}
Kallenberg, O. (2005) {\it Probabilistic symmetries and invariance principles}.   New York, Springer-Verlag.


\bibitem{}
Kesten, H. and Maller, R. A. (1994) Infinite limits and infinite limit points for random walks and trimmed sums. {\it Ann. Probab.} {\bf 22},
1475-1513


\bibitem{}  Logan, B. F., Mallows,  C. L., Rice,   S. O.  and Shepp,  L. A. (1973)
 Limit distributions of self-normalized sums. {\it Ann. Probab.} {\bf 1},  788-809.





\bibitem{} Martsynyuk, Yu.V.\ (2009a)
Functional asymptotic confidence intervals for the slope in linear error-in-variables models. {\it Acta Mathematica Hungarica}. {\bf 123}
133--168.


\bibitem{} Martsynyuk, Yu.V.\ (2009b)
Functional asymptotic confidence intervals for a common mean of independent random variables. {\it The Electronic Journal of Statistics}. {\bf
3} 25--40.

\bibitem{}
 O'Brien, G. L. (1980) A limit theorem for sample maxima and heavy branches in Galton-Watson trees.
 {\it J.   Appl. Probab.} {\bf 17}, 539-545.

\bibitem{} Petrov, V. V. (1995) {\it Limit Theorems of Probability Theory,
Sequences of Independent  Random Variables}.  Clarendon Press, Oxford.

\bibitem{} Sato, K.(1999) {\it L\'{e}vy Processes and Infinitely Divisible Distributions}. Cambridge University Press.

\bibitem{Sh97}  Shao, Q.-M. (1997)
Self-normalized large deviations. {\it Ann. Probab.} {\bf  25}
 285--328.


\bibitem{Sh98 }  Shao, Q.-M. (1998)
Recent developments on self-normalized limit theorems. In {\it Asymptotic Methods in Probability and Statistics. A Volume in Honour of
M.~Cs\"org\H{o}} (B. Szyszkowicz, ed.) 467--480.  North-Holland, Amsterdam


\bibitem{Shao2004}  Shao, Q.-M. (2004)
Recent progress on self-normalized limit theorems. {\it Probability, Finance and Insurance} (T.L.\ Lai, H.~Yang and S.P.~Yung, eds.). World
Scientific.


\bibitem{Shao97 }  Shao, Q.-M. (2010)
Stein's Method, Self-normalized Limit Theory and Applications. {\it Proceedings of the Int.\ Cong.\ of Mathematicians 2010 (ICM 2010)},
2325--2350.






\end{thebibliography}
\end{document}